\documentclass[11pt]{article}
\usepackage{mainsty}
\usepackage[utf8]{inputenc}
\usepackage{authblk}
\usepackage{upref}

\newcommand{\tim}[1]{\textcolor{orange}{#1}}

\title{Non-identifiability distinguishes Neural Networks among Parametric Models}
\author{Sourav Chatterjee, Timothy Sudijono\footnote{Department of Statistics, Stanford University} } 
\date{\today}

\begin{document}
\maketitle

\begin{abstract}
One of the enduring problems surrounding neural networks is to identify the factors that differentiate them from traditional statistical models. We prove a pair of results which distinguish feedforward neural networks among parametric models at the population level, for regression tasks. Firstly, we prove that for any pair of random variables $(X,Y)$, neural networks always learn a nontrivial relationship between $X$ and $Y$, if one exists. Secondly, we prove that for reasonable smooth parametric models, under local and global identifiability conditions, there exists a nontrivial $(X,Y)$ pair for which the parametric model learns the constant predictor $\E[Y]$. Together, our results suggest that a lack of identifiability distinguishes neural networks among the class of smooth parametric models.
\end{abstract}

\section{Introduction}

Consider a data distribution modeled by a random vector $X \in \bR^p$ and  a random variable $Y \in \bR$, representing the outcome. Suppose that there is a learnable relationship between the two random variables, which we encode by the condition 
\[
\E[\Var[Y | X]] < \Var[Y].
\]
Firstly, we show that in such settings, feedforward neural networks will always capture some of the variation. The first result of the paper proves that at the population level, the mean-squared error (MSE) of the best fitting neural network $\hat{f}$ is always less than $\Var(Y)$. The result is proved under minimal assumptions on the architecture of the neural network. 

Secondly, we provide a partial converse to our first result. This result shows that for most reasonable smooth parametric models, there exist distributions $(X,Y)$ such that the model learns nothing useful at all: the best-fitting model is the constant prediction $\E[Y].$ The differentiating conditions assumed in the converse result are what we call local identifiability and strong identifiability. The former condition requires the statistical model to have an invertible Fisher information matrix, as defined in Theorem \ref{thm:population_nfl}; the latter requires the model parameters to be close if the learned functions are close in $\bL^2$ norm.

Together, our results suggest that the lack of identifiability is a differentiating property for neural networks among the class of smooth parametric statistical models, in the sense of being able to learn a weak relationship in any nontrivial data distribution. This is surprising, as the statistics literature has traditionally focused on identifiability as a desirable property. 

To illustrate our results, we contrast the logistic regression model 
\begin{equation}
\label{eq:logistic}
f^{\sf{log}}_{\theta}(x) = \frac{e^{\alpha + \beta x}}{1+e^{\alpha + \beta x}}, \quad \theta = \left(\alpha, \beta \right) \in \bR^2
\end{equation}
with the one-layer neural network model 
\begin{equation}
\label{eq:nn}
f^{\sf{NN}}_\theta(x) = \gamma + \delta \frac{e^{\alpha + \beta x}}{1+e^{\alpha+\beta x}}, \quad \theta = \left(\alpha, \beta, \gamma, \delta \right) \in \bR^4
\end{equation}
The difference between the models is the presence of the two parameters $\gamma,\delta$. In this setting, Theorem \ref{thm:population_nfl} demonstrates the existence of a data distribution $(X,Y)$, for which the best-fitting logistic regression model in square loss is simply $\E[Y]$. Concretely, the function $f^{\log}_\theta$ of the form in Eq. \eqref{eq:logistic} which minimizes
\[
\E\left[(Y - f^{\sf{log}}_\theta(X))^2 \right]
\]
is given by the parameters $\alpha = \text{logit}(\E[X]), \beta = 0$. Meanwhile, Theorem \ref{thm:neural_networks_learn_everything} shows there exists some choice of parameters $\tilde{\theta}$ in the one-layer neural network model, yielding a function $f^{\sf{NN}}_{\tilde{\theta}}(x)$ which satisfies 
\[
\E\left[(Y - f^{\sf{NN}}_{\tilde{\theta}}(X))^2 \right] < \Var(Y) .
\]
We state this precisely in the following proposition when $X \sim \textsf{MVN}(0,I_p)$, a $p$-dimensional vector of i.i.d. Gaussians. It follows as a direct consequence of Theorem \ref{thm:neural_networks_learn_everything} and Corollary~\ref{cor:logistic_regression}. Section \ref{sec:discussion} discusses the differences between these two models in detail.

\begin{proposition}[Logistic model vs.~neural network model]
Let $X \sim \textup{\textsf{MVN}}(0,I_p)$. Then there exists a random variable $Y$ such that the best fitting logistic model minimizing square loss is the constant prediction $\E[Y]$. That is,
\[
\inf_\theta \E\left[(Y - f^{\sf{log}}_\theta(X))^2 \right] = \Var(Y).
\]
Explicitly, $Y$ is a random variable such that $\mathbb{E}(Y=1|X=x) = \frac{1}{2}+\epsilon g_0(x)$ for sufficiently small $\epsilon$, where $g_0$ is a bounded even function. On the other hand, for any random variable $Y'$ such that $\E[\Var(Y'|X)] < \Var(Y')$, we have 
\[
\inf_\theta \E\left[(Y' - f^{\sf{NN}}_\theta(X))^2 \right] < \Var(Y').
\]
\end{proposition}

\section{Results}

\subsection{Setup} We consider feedforward neural network architectures with $D+2$ layers, where the last layer is the output, and the first is the input layer. Denote the widths of each intermediate layer (i.e., all layers except the output layer) as $w_0,\dots,w_D$. The input layer takes in data in $\bR^{w_0}$, and the output layer is a single real number. The connection weights and biases for $d$th intermediate layer will be denoted by $W^{(d)} \in \bR^{w_d \times w_{d-1}}$ and $b^{(d)} \in \bR^{w_d}, 1 \leq d \leq D$. We will let $W^{(d)}_j$ denote the $j$th row of $W^{(d)}$, and $b_j^{(d)}$ denote the $j$th entry. Throughout the paper we will focus on neural network architectures for regression tasks so that the output layer is linear. Denote the weight matrix of the last layer by $W^{(D+1)} \in \bR^{w_D}$ (now a row vector), and the bias vector with $b^{(D+1)} \in \bR.$

We will say that an activation function $\sigma$ is of \textit{tanh form} of it satisfies 
\[
\sigma(0) = 1/2, \ \ \lim_{x \rightarrow  -\infty} \sigma(x)  = 0, \ \ \lim_{x \rightarrow  \infty} \sigma(x)  = 1,
\]
and is monotonically increasing. The commonly used hyperbolic tangent, exponential linear unit, and sigmoid functions are of this type, after appropriate scaling and shifting. The even more commonly used ReLU function $\sigma(x) = \max(x,0)$ is not of this type; but our result covers the ReLU function too. Altogether, the feedforward network is the model 
\[
f_\theta(x) = W^{(D+1)}\sigma(\dots W^{(2)}\sigma(W^{(1)}x + b^{(1)})  +b^{(2)}  \dots) + b^{(D+1)},
\]
parametrized by its connection weights and biases $\theta = \set{W^{(i)},b^{(i)}}_{1\le i\le D+1}.$  The notation means that the activation functions act component-wise on their input vectors. Finally, call $H^{(d)}(x)$ the vector-valued function representing the output of the hidden units in layer $d$, where $x \in \bR^{w_0}$. These functions are defined inductively as 
\[
H^{(d)}(x) = \sigma\left(W^{(d)}H^{(d-1)}(x) + b^{(d)}\right),
\]
where the nonlinearity is again applied component wise.

\subsection{Neural Networks Always Weakly Learn}

Fix a probability space $(\Omega,\cl{F},\Prob)$ and a pair of random variables $(X,Y)$ defined on it, where $X: \Omega \rightarrow \bR^p, Y: \Omega \rightarrow \bR$ such that
\[
\E[\Var(Y | X)] < \Var(Y).
\]
We will focus on a regression task where $Y$ is a general random variable in $\bR$ and the neural network is fitted with a linear output layer. Fix an architecture for the neural network, and let $\hat{f}$ denote the mean-square-optimal neural network with this fixed architecture. That is, let $\cl{F}$ denote the class of functions generated by the neural network as the learnable parameters vary. Define 
\[
\hat{f} := \argmin_{f \in \cl{F}} \E[(Y - f(X))^2],
\]
where the minimizer is chosen by some predetermined rule if there is more than one minimizer. 
Then, under minimal assumptions on the neural network architecture, the following theorem shows that $\E[(Y - \hat{f}(X))^2] < \Var(Y).$

\begin{theorem}
\label{thm:neural_networks_learn_everything}
Consider any pair of random variables $(X,Y)$ where $X \in \bR^p$ and $Y \in \bR$, such that $Y$ is square-integrable and $\E[\Var(Y|X)] < \Var(Y).$ Suppose that $X$ has density with respect to Lebesgue measure. Given any regressor neural network architecture with $d \geq 1$ hidden layers, a linear output, and activation function of tanh form, let $\hat{f}$ be any mean-square optimal neural network with this architecture, chosen by some predetermined rule if there is more than one optimizer. Then
\[
\E[(\hat{f}(X) - Y)^2] < \Var(Y).
\]
If the nonlinearity is ReLU, then the same conclusion holds, but $X$ need not have a density, and we require $w_D \geq 2$.
\end{theorem}


This theorem is proved in Section \ref{proofsec}. The idea of the proof is to notice that the condition $\E[\Var(Y | X)] < \Var(Y)$ implies the existence of a set $S$ such that $Y$ and $\mathbf{1}_S(X)$ are correlated. In the regression setting, we show that $S$ can be taken as a half-space on $\bR^p$. The key lemma is the following, which is closely related to the Cramer--Wold theorem. 

\begin{lemma}
\label{lemma:correlated_indicator_regression_case}
Suppose $Y$ is square-integrable and $\E[Y|X] \neq \E[Y]$ on a set of nonzero measure. Then there exists an event of the form $\set{\alpha \cdot X < t}$
such that $\E[Y|X]$ has nonzero correlation with the indicator random variable $\mathbf{1}\set{\alpha \cdot X < t}.$
\end{lemma}

The second step is to prove that neural networks can approximate half-space indicator functions, and affine transformations of these indicators. Taking the right linear combination $a\mathbf{1}_S(X) + b$ results in a predictor $\hat{f}$ with MSE better than that of the best constant prediction $\E[Y].$ Choosing the right constants $a,b$ is important, as otherwise there may not be an improvement in MSE. While regular GLMs may be able to fit $\mathbf{1}_S(X)$ arbitrarily well, neural networks can fit affine transformations of $\mathbf{1}_S(X)$. This is an important distinction.

\begin{lemma}[Approximation capabilities of neural networks]
\label{lemma:approximation_neural_networks}
Let $X$ be an $\bR^p$-valued random vector that has density with respect to Lebesgue measure. Fix parameters $\alpha, c_0,c_1,c_2,$ and consider any neural network with ReLU activation, $D \geq 1$ hidden layers, and a linear output layer with $w_D \geq 2,$ where recall that $w_D$ is the width of the last hidden layer of the network. Then there exist parameters $\theta$ such that $f_\theta(X)$ is arbitrarily close to the random variable $c_2\mathbf{1}\set{\alpha^\top X \leq c_1} + c_0$ in $\bL^2$-norm. The same result holds for tanh form activations, without the assumption that $w_D \geq 2.$
\end{lemma}

The approximation capability of neural networks is a topic of intense study in the literature on neural networks, going under the name of ``universal approximation results". Some of the first landmark results in this area were obtained around the same time by Hornik and Cybenko \cite{hornik1989multilayer, cybenko1989approximation}, who proved that infinitely-wide, single layer neural networks can approximate a wide class of functions arbitrarily well. An improvement on these results, known as Barron's theorem \cite{barron1993universal} provides quantitative bounds on the approximation power of these networks. 

However, these early results do not explain the importance of depth, a definite contributor to the outstanding success of modern neural networks. One strand of recent work has focused on depth separations. This line of research exhibits functions which are difficult to express with neural networks of depth $d$, but easy with neural networks of greater depth. See \cite{eldan2016power, telgarsky2015representation, delalleau2011shallow, lee2017ability} and the surrounding references for some explicit examples. Stronger quantitative results were obtained recently in \cite{hanin2019universal, hanin2017approximating} for the ReLU nonlinearity, where the depth necessary to approximate certain classes of functions was determined.

We do not mention the complementary optimization perspective. Even if neural networks can easily express a given function, it may be difficult to learn the function with algorithms like gradient descent. See  \cite{malach2021connection} and the references therein for recent progress.

\subsection{A Partial Converse to Theorem \ref{thm:neural_networks_learn_everything}}

In this section, we establish a kind of converse to Theorem \ref{thm:neural_networks_learn_everything}. This theorem says, roughly, that for most classical statistical learning models $\set{f_\theta}_{\theta \in \bR^d}$, there exist a pair of random variables $(X,Y)$ satisfying $\E[\Var(Y|X)] < \Var(Y)$ such that the mean-square optimal model $\hat{f}$ satisfies $\text{MSE}(\hat{f}) = \Var(Y).$ That is, the best fitting model is the constant prediction $\E[Y].$ The distinguishing factor here is \textit{identifiability}, interpreted in both a local and a global sense.  

In the following, $B(\theta_0, R)$ denotes the ball of radius $R$ centered at $\theta_0$,  $M\succ 0$ means that a matrix $M$ is positive definite, $\norm{\theta}_2$ denotes the Euclidean norm of a vector $\theta\in \bR^d$, and $\norm{Y}_{\bL^2}$ denotes the $\bL^2$-norm of a random variable $Y$.

\begin{theorem}
\label{thm:population_nfl}
Consider any collection of functions  $\set{f_\theta}_{\theta \in \bR^d}$ with domain $\bR^p$ and range $\bR$ and  which are $C^2$ in the $\theta$ argument. Suppose there exist $\theta_0$ with $f_{\theta_0}$ constant and a random vector $X \in \bR^p$ defined on a probability space $(\Omega,\cl{F},\Prob)$ such that 
\begin{enumerate}
    \item $I(\theta_0) = \E[\nabla_\theta f_{\theta_0}(X)\nabla_\theta f_{\theta_0}(X)^\top] \succ 0$,
    \item For every $\epsilon > 0$ less than some $\e_0$, there exists $\delta > 0$ such that $\norm{f_{\theta}(X) - f_{\theta_0}(X)}_{\bL^2} < \delta$ implies $\norm{\theta - \theta_0}_2 < \e$.
    \item For some $R$, 
    \begin{equation}
    \label{eq:integrability_condition}
    \sup_{\theta \in B(\theta_0,R)} \biggl|\pdv[2]{}{\theta_i}{\theta_j}f_{\theta}(X)\biggr| \in \bL^2(\Prob)
    \end{equation}
    for all $i,j$.
    \item The support of $X$ (denoted $\supp X$) has cardinality greater than $d+1$. 
\end{enumerate}
Then there exists a function $g$ such that $Y:= g(X)$ satisfies (1) $g(X)$ is not constant and (2) $\E[(Y-f_\theta(X))^2]$ is minimized at $\theta=\theta_0$. 
\end{theorem}
The first condition can be thought of as a \textit{local identifiability condition}. We will refer to the matrix $I(\theta_0)$ as the Fisher information matrix; it is indeed the Fisher information when the outputs of the neural network are corrupted by iid Gaussian noise.  If the Fisher information were not invertible, there would exist directions in parameter space where the output of the model would not change. The second condition we term \textit{strong identifiability}, because if it does not hold, then one can find two sequences of $\theta$'s are bounded away from each other but the corresponding $f_\theta$'s converge.

\subsection{Examples}
Before we discuss the implications for neural networks and the interaction with Theorem~\ref{thm:neural_networks_learn_everything}, let us showcase a few models which satisfy the assumptions of Theorem \ref{thm:population_nfl}. Linear regression with features $T_1(\cdot),\dots,T_d(\cdot)$ generally satisfies the assumptions of Theorem \ref{thm:population_nfl}, under some non-collinearity assumptions on $T.$ The model is parametrized by the vector $\beta \in \bR^d$, with 
\begin{equation}
\label{eq:linear_model}
f^{\textsf{lin}}_\beta(x) = \sum_{i=1}^d \beta_i T_i(x), \quad x\in \bR^p.
\end{equation}

\begin{corollary}[Linear Regression]
\label{cor:example_linear_regression}
Consider any random vector $X\in \bR^p$ with $|\supp X| > d+1$, such that the $d\times d$ matrix $\Sigma$ with entries $\Sigma_{ij} = \E[T_i(X)T_j(X)]$ is invertible, Then there exists a non-constant function $g$ such that for the data distribution $(X,g(X))$, the best-fitting model in square loss is $f^{\textsf{lin}}_\beta(x) = 0$, for $\beta = 0.$
\end{corollary}

By orthogonality, the result can be extended to random variable $Y$ such that $\E[Y|X] = g(X)$ above. A full proof can be found in Section \ref{proofsec}. \\

Logistic Regression provides another example. Specifically, models of the form 
\[
f^{\textsf{log}}_\theta(x) = \frac{\exp(\beta_0 + \beta^\top x)}{1 + \exp(\beta_0 + \beta^\top x)}, \quad x \in \bR^p
\] 
satisfy the conditions of Theorem \ref{thm:population_nfl}, parametrized by $\theta = (\beta_0,\beta) \in \bR^{p+1}.$ 

\begin{corollary}[Logistic Regression.]
\label{cor:logistic_regression}
Take $X$ to be the random vector with iid Gaussian coordinates and consider $\theta_0 = (\beta_0,\beta)$ with $\beta_0 = 0, \beta = 0.$  Consider any random variable $Y$ such that $\E[Y|X] = \frac{1}{2} + \e g_0(X)$ with a bounded even function $g_0$ with $\E[g_0(X)] = 0$, and $\e$ small enough. Then for the data distribution $(X,Y)$, the best-fitting logistic model in square loss is $f_{\theta_0}^{\sf{log}} = 1/2$. 
\end{corollary}

The result can likely be extended to other random vectors with suitable anti-concentration properties, as the proof shows.

\section{Discussion}
\label{sec:discussion}
Theorem \ref{thm:neural_networks_learn_everything} shows that neural networks do not satisfy the conditions of Theorem \ref{thm:population_nfl}. It is clear that the identifiability conditions fail for neural networks, and therefore distinguish neural networks from standard parametric models. As an illustration, we revisit the example given in the introduction, concerning the logistic model and the one-layer neural network models with a one-dimensional covariate.
\begin{align*}
f^{\sf{log}}_{\theta}(x) & = \frac{e^{\alpha + \beta x}}{1+e^{\alpha + \beta x}}, \quad \theta = \left(\alpha, \beta \right) \in \bR^{2} \\
f^{\sf{NN}}_\theta(x) & = \gamma + \delta \frac{e^{\alpha + \beta x}}{1+e^{\alpha+\beta x}}, \quad \theta = \left( \alpha, \beta, \gamma, \delta \right) \in \bR^{4}.
\end{align*}
Due to the presence of the additional parameters $\gamma,\delta$ the neural network models are no longer identifiable. For example, taking the parameters $(\alpha,\beta,\gamma,\delta) = (0,0,0,1)$ and $(\alpha',\beta',\frac{1}{2},0)$ for any $\alpha',\beta'$ both yield the same function $f^{\sf{NN}}_\theta(x) = 1/2.$ Thus the strong identifiability condition in Theorem \ref{thm:population_nfl} fails. In addition, the local identifiability condition fails. To illustrate this, we can compute the gradients of both function classes with respect to their parameters. Let $\sigma(x) := \frac{e^x}{1+e^x}, m_1(x) := \sigma(\alpha + \beta x),$ and $m_2(x) := \sigma'(\alpha + \beta x)$. Then
\begin{align*}
\nabla_\theta f_\theta^{\sf{log}}(x) & = \left(m_2(x), x m_2(x) \right) \\
\nabla_\theta f_\theta^{\sf{NN}}(x) & = \left(\delta m_2(x), \delta x m_2(x),1,m_1(x) \right).
\end{align*}

Then 
\[
\E[\nabla_\theta f_\theta^{\sf{log}}(X) \nabla_\theta f_\theta^{\sf{log}}(X)^\top]=
\begin{bmatrix}
   \E[m_2^2(X)] & \E[X m_2^2(X)] \\ 
   \E[X m_2^2(X)] & \E[X^2 m_2^2(X)] 
\end{bmatrix}.
\]
Let $\bb{Q}$ be the measure with Radon--Nikodym Derivative $\frac{d\bb{Q}}{d\Prob}(x) = m_2^2(x).$ Then the Fisher information matrix for the logistic model is invertible whenever $X$ is non-constant under $\bb{Q}.$ On the other hand, the Fisher information matrix for the neural network model is not invertible for a wide range of parameters. To see this, consider $\beta = 0,$ in which case $m_2,m_1$ are constants. This implies the Fisher information matrix for the neural network model is not invertible, as the last two elements of $\nabla_\theta f_\theta^{\sf{NN}}(x)$ are constants.

For multilayer feedforward neural networks identifability generally fails to hold. The strong identifiability condition fails for symmetry reasons: one could permute the weights between neurons in each layer and the output variable would still be the same. The scale homogeneity of the ReLU function can also be leveraged to show lack of strong identifiability; \cite{bui2020functional} shows that for some ReLU architectures, permutation and scaling are the only function-preserving weight transformations. Theorem 4.2 of \cite{petersen2021topological} proves a related stronger result: neural networks close in the $L^\infty$ function norm can have no parametrizations with close weights. Regarding the local identifiability condition, it turns out that questions about the spectrum of the Fisher information matrix are of great interest in the deep learning literature. For example, \cite{fukumizu1996regularity} gives a sufficient condition for the Fisher information matrix (for any positive continuous density) of deep neural networks to be strictly positive definite, in terms of a criterion called \textit{irreducibility}. The field of singular learning theory \cite{watanabe2009algebraic, watanabe2018mathematical, wei2022deep, amari2003learning} has investigated these questions in considerable detail. See \cite{wei2022deep, farrugiaroderts2022structural} and the references therein for results on the singularity of neural network models. Related work includes \cite{pennington2018spectrum}, which analyzes a single hidden layer neural network in the population limit. The paper studies the infinite depth limit, with Gaussian data and random Gaussian parameters. They characterize the spectrum completely, showing that it is well separated from zero. 

The main interest in the spectrum is to prove a long-standing conjecture about the structure of the empirical Fisher information: most of its eigenvalues are bulked together near zero while there are a few extremely large ones, which are known to cause issues in optimization. For example \cite{karakida2019universal, karakida2019pathological, karakida2019normalization} study the spectrum of the Fisher information in deep neural networks in the mean field limit. This work and many other empirical works lend support to this rough picture for the spectrum \cite{papyan2019measurements, ghorbani2019investigation}. All of these works analyze the model with finite samples, not at the population level. Related to this literature are the studies on the loss landscape of deep neural networks. Many of these papers focus on the idea of `flat' minima as critical to generalization \cite{sagun2017empirical,dinh2017sharp}.

Separately, Theorem \ref{thm:population_nfl} has a passing resemblance to No Free Lunch theorems \cite{wolpert1996lack, shalev2014understanding}, which show that for any learning algorithm, there exists data distributions for which the algorithm generalizes poorly to out-of-sample data. However, Theorem \ref{thm:population_nfl} is instead a population-level statement about the best-fitting model in a class of models.

\section{Proofs}\label{proofsec}

\begin{proof}[Proof of Lemma \ref{lemma:correlated_indicator_regression_case}]
Suppose for contradiction that 
\[
\E[\E[Y|X]\mathbf{1}\set{\alpha \cdot X < t}] = \E[Y] \Prob(\alpha \cdot X < t)
\]
for every choice of $\alpha,t$. Using the functional monotone class theorem, we see that 
\[
\E[\E[Y|X]f(\alpha \cdot X)] = \E[Y] \E[f(\alpha \cdot X)]
\]
for any $f:\bR \rightarrow \bR$ bounded and measurable. Indeed, let $\cl{H}$ be the set of functions $h$ such that 
\[
\E[\E[Y|X]h(\alpha \cdot X)] = \E[Y] \E[h(\alpha \cdot X)],
\]
and $\cl{H}_+ := \set{h \in \cl{H}: h > 0}.$
Notice that $\cl{H}$ is a vector space, contains the constant function $1$, and $\cl{H}_+$ is closed under bounded increasing limits by the dominated convergence theorem. By supposition, $\cl{H}$ contains the set of functions $\cl{K} := \set{ c\mathbf{1}\set{x < t},\forall c,t}$. Applying the functional monotone class theorem, we find that $\cl{H}$ contains all bounded measurable functions. In particular, 
\[
\E[\E[Y|X]e^{i\alpha \cdot X}] = \E[Y] \E[e^{i\alpha \cdot X}],
\]
where $i = \sqrt{-1}$. Let $\mathcal{F}$ be the set of all $g\in \bL^1(\bR^p)$ such that its Fourier transform $\hat{g}$ is in $\bL^1(\bR^p)$. Take any $g\in \mathcal{F}$. Multiplying both sides of the above display by $\hat{g}(\alpha)$, integrating both sides with respect to $\alpha$, and switching order of integration by Fubini's theorem, we conclude that 
\begin{equation}
\label{eq:no_correlation_fourier}
\E[\E[Y|X]g(X)] = \E[Y]\E[g(X)],    
\end{equation}
by the Fourier inversion theorem. Fubini's theorem is justified since $\widehat{g} \in \bL^1(\bR^p), \E[Y|X] \in \bL^1(\Prob).$
By approximation, we will show that for any rectangle $A = \bigotimes_{i=1}^p A_i,$ with $A_i := [a_i,b_i]$,
\begin{equation}
\label{eq:rectangle_no_corr}
\E[\E[Y|X]\mathbf{1}_{A}(X)] = \E[Y]\E[\mathbf{1}_{A}(X)].
\end{equation}
Since the set of all rectangles generates the Borel sigma algebra, we conclude from the definition of the conditional expectation that $\E[Y|X] = \E[Y]$.

To show Eq. \eqref{eq:rectangle_no_corr}, first notice that $C_c^\infty(\bR^p)$ is contained in $\cl{F}$, see e.g. \cite{folland1999real}. For each $i$, take a smooth bump function $\psi_{i,\e}$ which is equal to $1$ on $A_i$, zero outside $A_{i,\e} := [a_i - \e,b_i + \e]$, and is between $[0,1]$ on $A_{i,\e}.$ Then the function $\phi_{A,\e} := \prod_{i=1}^p \psi_{i,\e}(x_i)$ from $\bR^p $ to $\bR$ is contained in $C_c^\infty(\bR^p)$, bounded in $[0,1]$, and approaches $\mathbf{1}_A(x)$ pointwise as $\e \rightarrow 0.$ Then, because $\phi_{A,\e}$ and $\mathbf{1}_A$ are bounded and $\E[Y|X] \in \bL^1$, dominated convergence theorem shows that
\begin{align*}
     \E[\E[Y|X]\phi_{A,\e}(X)] & \xrightarrow{\e\downarrow 0} \E[\E[Y|X]\mathbf{1}_{A}(X)] \\
    \E[\phi_{A,\e}(X)] & \xrightarrow{\e\downarrow 0} \E[\mathbf{1}_{A}(X)] 
\end{align*}
Using Eq. \eqref{eq:no_correlation_fourier} on the functions $\phi_{A,\e}$, we conclude Eq. \eqref{eq:rectangle_no_corr} as desired.
\end{proof}

\begin{proof}[Proof of Lemma \ref{lemma:approximation_neural_networks}]
We will show the statement in two cases, depending on the type of the activation function. Let us first suppose the activation function $\sigma(x)$ is of tanh type. Set $W_1^{(1)} = k\alpha$, $b_1^{(1)} = -kc_1$, and other parameters in the first layer $[W_j^{(1)}, b_j^{(1)}, j =2,\dots,w_1]$ to zero. Recall that $W^{(1)}_j$ represents the $j$th row of the first layer weight matrix, and $b_j^{(1)}$ is the $j$th component of the bias vector. The parameter $k$ will be made arbitrarily large in the end. Thus the first unit in the first hidden layer $H_1^{(1)}(x)$ models the function $\sigma(k(\alpha^\top x - c_1))$. 

Now, set the connection weights to be $W_1^{(d)} = e_1$ for all $d \geq 2$ and all other parameters to zero. Let $\sigma^{(D)}$ represent the $D$-fold composition of the function $\sigma$. Then the first neuron in the last hidden layer models the function
\[
\sigma^{(D)}(k(\alpha^\top x - c_1)).
\]
Take the last layer weight vector as $-c_2e_1$ and the bias constant as $c_0+c_2$. With these parameters, the neural network represents the function
\[
c_2\left(1-\sigma^{(D)}(k(\alpha^\top x - c_1)) \right) + c_0.
\]
Let $Z := \alpha^\top X - c_1$. Then it suffices to show that as $k \rightarrow \infty$,
\[
\E\left[\left(\sigma^{(D)}(kZ) - \mathbf{1}\set{Z > 0} \right)^2 \right] \rightarrow 0.
\]
This follows by dominated convergence theorem: $\sigma^{(D)}(kx)$ pointwise converges to $\mathbf{1}\set{x > 0}$ except at $x = 0$, $Z$ has a density, and both $\sigma^{(D)}$ and $\mathbf{1}\set{x > 0}$ are bounded.

Next, suppose that our activation $\sigma(x)$ is the ReLu function. Let us make two observations. First, $\sigma(\sigma(x)) = \sigma(x).$ Second, the function $g(x) := \sigma(x) - \sigma(x-1)$ takes the value $1$ on $x \geq 1,$ is zero on $x \leq 0$, and is linear in between. This implies that $1 - g(kx) \to\mathbf{1}\set{x \leq 0}$ pointwise as $k\to \infty$.  Take $W_1^{(1)} = k\alpha$,  $b_1^{(1)} = -kb_1$, and the other layer weights and biases to be $W_1^{(d)} = e_1, b^{(d)} = 0$ for all $d \geq 2,\dots,D$. Finally, set $b_2^{(D)} = -1,W_2^{(D)} = e_1$, so that $H_1^{(D)} = \sigma^{(D)}(k(\alpha^\top x - c_1)), $ and $H_2^{(D)} = \sigma^{(D)}(k(\alpha^\top x - c_1) - 1).$


Set the final layer weights as $W^{(D+1)} = c_2e_2 - c_2e_1$, with $b^{(D+1)} = c_0 + c_2.$ Then the neural network equals the function
\[
c_2\left[1 - g\left(k\left(\alpha^\top x - c_1 \right) \right)\right] + c_0.
\]
Let $Z := \alpha^\top X - c_1.$ As $k \rightarrow \infty$, by the dominated convergence theorem,
\begin{align*}
    & \E[\left(c_2 \left[ 1 - g\left(kZ\right) \right] + c_0 - c_2\mathbf{1}\set{Z \leq 0} - c_0 \right)^2] \\
    &=  c_2^2 \E\left[\left( g(kZ) - \mathbf{1}\set{Z > 0}\right)^2 \right] \rightarrow 0.
\end{align*}
\end{proof}

\begin{proof}[Proof of Theorem \ref{thm:neural_networks_learn_everything}]
Because $\E[\Var(Y|X)] < \Var(Y)$, it must be that $\E[Y|X] \neq \E[Y]$ with positive probability. Lemma \ref{lemma:correlated_indicator_regression_case} implies the existence of parameters $\alpha, b$ such that $\E[Y|X]$ is correlated with $\mathbf{1}\set{\alpha^\top X + b \leq 0}.$ Let $A := \set{\alpha^\top X + b \leq 0}.$ Consider the best linear predictor of $Y$ based on $\mathbf{1}_A$,
\begin{equation}
\label{eq:best_linear_predictor}
c_1\mathbf{1}_A + c_0,
\end{equation}
given by
\begin{align*}
    c_1 & = \frac{\Cov(Y,\mathbf{1}_A)}{\Var(\mathbf{1}_A)}, \\
    c_0 & = \E[Y] - c_1 \Prob(A).
\end{align*}
The mean squared error of this best linear predictor can be explicitly calculated as 
\[
\Var(Y) - \frac{\Cov(Y,\mathbf{1}_A(X))^2}{\Var(\mathbf{1}_A(X))}.
\]
By Lemma \eqref{lemma:approximation_neural_networks}, there exist parameters for which the neural network architecture fits Equation \ref{eq:best_linear_predictor} arbitrarily well. 
That is, for any $\e > 0$, there exists parameters $\theta = \set{W^{(i)}, b^{(i)}}_{1 \leq i \leq D+1}$ such that the neural network $f^{\textsf{NN}}_\theta$ satisfies 
\[
\E[\left(f_\theta(X) - c_1\mathbf{1}_A(X) - c_0\right)^2] \leq \e.
\]
Therefore, 
\[
\text{MSE}(\hat{f}) < \Var(Y),
\]
as the covariance of $Y$ and $\mathbf{1}_A(X)$ is nonzero.
\end{proof}

\subsection{Proof of the converse theorem}

\begin{proof}[Proof of Theorem \ref{thm:population_nfl}]
    Let $f_{\theta_0} = c$. For any $f_\theta(X), \theta \in \bR^d$, define its \textit{linearization} 
    \[
    Lf_\theta(X) := f_{\theta_0}(X) + (\theta - \theta_0)^\top \nabla_\theta f_{\theta_0}(X).
    \]
    Because $|\supp X| > d+1$, the vector space consisting of functions $f(X) \in \bL^2(\Prob)$ has dimension greater than $d+1$. Thus, by the Gram--Schmidt algorithm, we may pick a nonconstant function $h(X)$ orthogonal to the collection $\set{1,\partial_{1} f_{\theta_0}(X),\dots,\partial_{d} f_{\theta_0}(X)}$, where $\partial_i f_{\theta}(x)$ denotes the partial derivative of $f$ in $\theta_i$. Scale $h(X)$ such that $\norm{h(X)}_{\bL^2} = \e$, for a parameter $\e$ that we will choose later. Define $g(X) := f_{\theta_0}(X) + h(X) = c + h(X)$. Because $h(X)$ is orthogonal to $1$, it has mean zero, and so $\Var g(X) = \e^2$. The random variable $g(X)$ will be our candidate for $Y$, for $\e$ chosen small enough. 

    We will show that the $\bL^2$ projection of $g(X)$ onto the set $\cl{S}_X := \set{f_\theta(X), \theta \in \bR^d}$ is $f_{\theta_0}.$ To begin, notice that $\|g(X)-c\|_{\bL^2} = \e$ by construction. For any other $f_\theta(X) \in \cl{S}_X$ we must prove that 
    \[
    \|f_\theta(X)-g(X)\|_{\bL^2} > \e.
    \]
    Notice that we need only consider $f_\theta$ such that $\|f_\theta(X)- c\|_{\bL^2} \leq 2\e$, because otherwise, we already have 
    \[
    \|f_\theta(X)-g(X)\|_{\bL^2} \geq \|f_\theta(X)- c\|_{\bL^2} - \|g(X)- c\|_{\bL^2} > \e.
    \]
    Throughout the rest of the proof, we will restrict to only working with $\theta$ satisfying the above constraint. By the triangle inequality, note that
    \begin{equation}
    \label{eq:popnfl_1}
         \|g(X)-f_\theta(X)\|_{\bL^2} \geq \|g(X)- Lf_\theta(X)\|_{\bL^2} - \|f_\theta(X)- Lf_\theta(X)\|_{\bL^2}.
    \end{equation}
    By our choice of $h(X)$, we have $g(X)-c$ and $Lf_\theta(X)-c$ are orthogonal. Thus: 
    \begin{equation}
    \label{eq:popnfl_2}
    \|g(X)- Lf_\theta(X)\|_{\bL^2} = \sqrt{\|g(X)-c\|_{\bL^2}^2 + \|c- Lf_\theta(X)\|_{\bL^2}^2}.
    \end{equation}
    Next, we claim that 
    \begin{equation}\label{eq:toshow}
    \|f_\theta(X)-Lf_\theta(X)\|_{\bL^2} \leq C_1 \|c-Lf_{\theta}(X)\|_{\bL^2}^2
    \end{equation}
    for a positive constant $C_1$ that does not depend on $\theta$.  To show this, first note that by Taylor approximation, we have
    \[
    f_\theta(x) = f_{\theta_0}(x) + (\theta - \theta_0)^\top \nabla_\theta f_{\theta_0}(x) + R_x(\theta - \theta_0),
    \]
    where we may write 
    \[
    R_x(\theta - \theta_0) = (\theta - \theta_0)^\top H_\theta f_{\theta_0 + c_x(\theta - \theta_0)}(\theta - \theta_0).
    \]
    for some constant $c_x \in [0,1],$ and where $H_\theta f_{\theta'}$ is the Hessian of $f$ at $\theta'.$ As a result,
    \begin{gather*}
        \|f_\theta(X) - Lf_\theta(X)\|_{\bL^2} = (\E R_X(\theta - \theta_0)^2)^{1/2},
    \end{gather*}
    and
    \begin{gather*}
        |R_x(\theta - \theta_0)| \leq \norm{H_\theta f_{\theta_0 + c_x(\theta - \theta_0)} }_{\text{op}} \norm{\theta - \theta_0}_2^2.
    \end{gather*}
    Separately, note that
    \begin{gather*}
        \|c- Lf_\theta(X)\|_{\bL^2}^2 = (\theta - \theta_0)^\top I(\theta_0)(\theta - \theta_0),
    \end{gather*}
    where $I(\theta_0)_{ij} = \E[\partial_{i} f_{\theta_0}(X)\partial_{j} f_{\theta_0}(X)].$ 
    
    Now, let $R$ be as in the statement of the theorem. When $\norm{\theta - \theta_0}_2 < R,$ we can bound $\norm{H_\theta f_{\theta_0 + c_x(\theta - \theta_0)}}_{op}$ by 
    \[
    C_d \max_{i,j} \sup_{\theta \in B(\theta_0, R)}|\partial^2_{\theta_i, \theta_j} f_{\theta}(x)|
    \]
    for some constant $C_d$ depending on $d$. By the third assumption of the theorem, the above quantity is square integrable. Meanwhile, 
    \begin{equation}\label{eq:positive}
    \|Lf_\theta(X) - c\|_{\bL^2}^2 \geq \norm{\theta - \theta_0}^2_2\lambda_{\min}(I(\theta_0)),
    \end{equation}
    where $\lambda_{\min}(I(\theta_0))$ denotes the smallest eigenvalue of $I(\theta_0)$. From this we obtain 
    \[
    \|f_\theta(X) - Lf_\theta(X)\|_{\bL^2} \leq C_1\|Lf_{\theta}(X) - c\|_{\bL^2}^2
    \]
    with 
    \[
    C_1 = \frac{C_d\norm{\max_{i,j} \sup_{\theta \in B(\theta_0, R)}|\partial_i \partial_j f_{\theta}(X)|}_{\bL^2}}{\lambda_{\min}(I(\theta_0))}, 
    \]
    which is finite by our assumptions. This proves Equation~\eqref{eq:toshow}. Combining Equations \eqref{eq:popnfl_1}, \eqref{eq:popnfl_2} and \eqref{eq:toshow}, we get
    \begin{equation}\label{eq:diffbd}
    \|g(X)-f_\theta(X)\|_{\bL^2} \geq \sqrt{\e^2 + \|c- Lf_\theta(X)\|_{\bL^2}^2} - C_1\|c-Lf_{\theta}(X)\|_{\bL^2}^2.
    \end{equation}
    Recall that $\e$ remains to be chosen. If we could choose $\e$ such that that $\e \leq 1/4C_1$, and 
    \begin{equation}\label{eq:maincond}
    \|c-Lf_\theta(X)\|_{\bL^2} < \sqrt{\frac{1 - 2\e C_1}{C_1^2}},
    \end{equation}
    the proof would be complete, by the following reasoning. Using equation \eqref{eq:positive} and the fact that $\theta \neq \theta_0$, we see that $\|c-Lf_\theta(X)\|_{\bL^2}>0$. Thus, squaring both sides of \eqref{eq:maincond} and multiplying by $C_1^2\|c-Lf_\theta(X)\|_{\bL^2}^2$, we get
    \[
    C_1^2 \|c-Lf_\theta(X)\|_{\bL^2}^4 < (1-2\e C_1) \|c-Lf_\theta(X)\|_{\bL^2}^2,
    \]
    which can be rearranged to give
    \[
    (C_1 \|c-Lf_\theta(X)\|_{\bL^2}^2 + \e)^2 < \|c-Lf_\theta(X)\|_{\bL^2}^2 + \e^2.
    \]
    This shows that the right side of Equation \eqref{eq:diffbd} is strictly bigger than $\e$.\\
    
    Thus, to complete the proof, we must ensure that  1) $\norm{\theta - \theta_0} \leq R,$ 2) $\e \leq 1/4C_1,$ and 3) that condition \eqref{eq:maincond} holds. Using the bound 
    \[
    \|Lf_\theta(X) -c\|_{\bL^2} \leq \lambda_{\max}(I(\theta_0))\norm{\theta - \theta_0}_2,
    \]
    we see that the first and third conditions are satisfied if
    \[
    \norm{\theta - \theta_0}_2 \leq \min\biggl\{R, C_3\sqrt{\frac{1 - 2\e C_1}{C_1^2}}\biggr\},
    \]
    where $C_3 := 1/\lambda_{\max}(I(\theta_0))$. 
    By the strong identifiability condition, there exists a $\delta>0$ so small that whenever $\|f_\theta(X)- c\|_{\bL^2} \leq \delta$, we have 
    \[
    \norm{\theta - \theta_0}_2 \leq\min\biggl\{\e_0, R, \frac{C_3}{\sqrt{2}C_1}\biggr\}.
    \]
    Pick any $\e \leq \min\{\e_0,\delta/2,1/4C_1\}.$ Then $\|f_\theta(X)-c\|_{\bL^2}\leq 2\e$ implies that $ \norm{\theta - \theta_0}_2 \leq R$ and also 
    \[
     \norm{\theta - \theta_0}_2 \leq \frac{C_3}{\sqrt{2}C_1} = C_3\sqrt{\frac{1 - \frac{2}{4C_1} C_1}{C_1^2}} \leq C_3\sqrt{\frac{1 - 2\e C_1}{C_1^2}}.
    \]
    Thus, $\|g(X)-f_\theta(X)\|_{\bL^2} >\e$. This completes the proof.
\end{proof}

\begin{proof}[Proof of Cor. \ref{cor:example_linear_regression}]
In this setting, $\nabla_{\beta} f^{\textsf{lin}}_\beta(x) = (T_1(x),\dots,T_d(x))$. Let us take $\beta_0 = 0.$ The local identifiability condition is immediately satisfied since the covariance matrix $\Sigma$ is equal to $I(\theta)$, and is invertible. Similarly, the strong identifiability condition is satisfied since 
\[
\norm{f^{\textsf{lin}}_\beta(X) - f^{\textsf{lin}}_{\beta'}(X)}_{\bL^2}^2 = (\beta - \beta')^\top \Sigma (\beta - \beta') \geq s_d(\Sigma) \norm{\beta - \beta' }_2^2 ,
\]
where $s_d(\Sigma) > 0$ is the smallest singular value of $\Sigma$. Thus $\norm{f^{\textsf{lin}}_\beta(X) - f^{\textsf{lin}}_{\beta'}(X)}_{\bL^2} \rightarrow 0$ implies $\norm{\beta - \beta' }_2 \rightarrow 0.$ Finally the Hessian of the model is simply zero, so the integrability condition \eqref{eq:integrability_condition} is satisfied.
\end{proof}

\begin{proof}[Proof of Cor. \ref{cor:logistic_regression}]
By projection properties of the conditional expectation,
\[
\E\left[(Y - f_{\theta_0}^{\textsf{log}}(X))^2 \right] = \E\left[(Y - \E[Y|X])^2 \right] + \E\left[(\E[Y|X] - f_{\theta_0}^{\textsf{log}}(X))^2 \right].
\]
Thus we may consider fitting the distribution $(X,\E[Y|X])$. We check the conditions of Theorem \ref{thm:population_nfl}.

The gradient $\nabla_\theta f_{\theta_0}^{\textsf{log}}$ can be calculated; it is equal to $(1/4, x/4).$
In this case, the Fisher information can be calculated to be 
\[
I(\theta_0) = 
\frac{1}{16}\begin{bmatrix}
    1 & \E[X^\top]\\
    \E[X] & \E[XX^\top] \\
\end{bmatrix} = I_{p+1}.
\] The second derivatives of $f_{\theta}^{\sf{log}}$ in $\theta$ are also bounded everywhere, so the integrability of the second derivatives is satisfied. To show strong identifiability around $1/2$, we need to show that for every $\e$, there exists $\delta$ such that 
\begin{equation}
\label{eq:log_cor_proof}
\norm{\frac{1}{2} - \frac{1}{1 + \exp(\alpha^\top X + \beta)}}_{\bL^2} \leq \delta \Rightarrow \norm{(\alpha,\beta)}_2 \leq \e.
\end{equation}
Fix $\e$ and suppose now that the left hand side of Eq. \eqref{eq:log_cor_proof} holds for $\delta$ to be chosen later. By Markov's inequality after squaring, we have that 
\begin{align*}
& \Prob\biggl(\biggl|\frac{1}{2} - \frac{1}{1 + \exp(\alpha^\top X + \beta)}\biggr| \geq 2\delta\biggr) \leq \frac{1}{4},
\end{align*}
which can be rewritten as 
\begin{align}\label{alphaineq}
& \Prob\biggl(|\alpha^\top X + \beta| > \log \frac{1+4\delta}{1 - 4\delta}\biggr) \leq \frac{1}{4}
\end{align}
since $1/2 - (1+e^x)^{-1}$ is an odd function and invertible. As a consequence of the above inequality, we claim that 
\[
|\beta| \leq \log \frac{1+4\delta}{1 - 4\delta}.
\]
Let $C := \log \frac{1+4\delta}{1 - 4\delta}$.  By symmetry, $\Prob(\alpha^\top X \leq 0) = 1/2$. Then if $\beta > C$, $\alpha^\top X \geq 0 \Rightarrow \alpha^\top X + \beta \geq C$, which implies that $\frac{1}{2} \leq \Prob(|\alpha^\top X + \beta| \geq C)$. A similar statement can be made if $\beta < -C$. This yields a contradiction to equation \eqref{alphaineq}. Combining this bound on $\beta$ with equation \eqref{alphaineq}, we conclude that 
\[
\Prob\biggl(|\alpha^\top X| > 2 \log \frac{1+4\delta}{1 - 4\delta}\biggr) \leq \frac{1}{4}.
\]
From this we infer that $\alpha$ must also be small. In particular, 
\[
2\Phi\left( -\frac{2}{\norm{\alpha}_2} \log \frac{1+4\delta}{1 - 4\delta} \right) \leq \frac{1}{4},
\]
which implies that 
\[
\norm{\alpha}_2 \leq c\log \frac{1+4\delta}{1 - 4\delta}
\]
for some positive constant $c.$ Combined with the bound on $\beta$, we see that  $\norm{(\alpha,\beta)}_2 \leq \e$ for $\e$ going to zero as $\delta \rightarrow 0.$ 

Inspecting the proof of Theorem \ref{thm:population_nfl}, we can take $\E[Y|X] = f_{\theta_0}^{\textsf{log}}(X) + \e g_0(X)$, for any $g_0(X)$ orthogonal to $1$ and $\partial_{\theta_i}f_{\theta}^{\textsf{log}}(X),$ for each $i=1,\dots,p.$ As calculated above, these partial derivatives are given by $(1/4,X/4)$. Take $g_0$ a bounded even function such that $\E[g_0(X)] =0$. By definition $g_0(X)$ is orthogonal to $1$, and symmetry shows that $\E[X_ig_0(X)] = 0$ for each $i$. Thus the construction in the proof of Theorem \ref{thm:population_nfl} allows us to take $\E[Y|X] = \frac{1}{2} + \e g_0(X)$ for small enough $\e.$  
\end{proof}

\section{Acknowledgments}

We thank Matthew Farrugia-Roberts for helpful comments and references. SC was supported by NSF grants DMS-2413864 and DMS-2153654. TS acknowledges support from the NSF Graduate Research Fellowship Program under Grant DGE-1656518.

\bibliographystyle{alpha}
\bibliography{main}

\end{document}